\documentclass[11pt,reqno]{amsart}

\usepackage{mathrsfs}

\newtheorem{lemma}{Lemma}

\newtheorem{corollary}{Corollary}
\newtheorem{proposition}{Proposition}
\newtheorem{definition}{Definition}
\newtheorem{remark}{Remark}

\begin{document}

\title{Orthogonal multiplications of type $[3,4,p], p\leq 12$}
\author{Quo-Shin CHi}
\address{Department of Mathematics, Washington University, St, Louis, MO 63130, USA}
\email{chi@math.wustl.edu}
\author{Haiyang Wang}
\thanks{The second author was partially supported by the Fundamental Research Funds for the Central Universities, China}
\address{School of Mathematical Sciences, Laboratory of Mathematics and Complex Systems, Beijing Normal University, Beijing 100875, P.R. China}
\email{wanghy15@mail.bnu.edu.cn}

\begin{abstract} We describe the moduli space of orthogonal multiplications of type $[3,4,p], p\leq 12,$ and its application to the hypersurface theory.
\end{abstract}

\maketitle

%%\keywords{}
%%\subjclass{Primary 11E25, 15B48}

\section{Introduction} An orthogonal multiplication of type $[m,n,p], m\leq n,$ is a bilinear map $F:{\mathbb R}^m\times {\mathbb R}^n\rightarrow {\mathbb R}^p$ such that $|F(x,y)|=|x||y|$ for all $x$ and $y$. The orthogonal multiplication is {\em full}\, if the image of $F$ spans ${\mathbb R}^p$; it follows that necessarily $n\leq p\leq mn$ when $F$ is full. Alternatively, we may define $x\circ y:=F(x,y)$, so that with respect to $\circ$, the orthogonal multiplication $F$ may be thought of as a sort of norm-preserving product among the involved Euclidean spaces. The most well known of such ``products'' may be the ones of the normed algbras ${\mathbb R}, {\mathbb C}, {\mathbb H},$ and the octonion algebra ${\mathbb O}$. Conversely, Hurwitz~\cite{H} proved, up to domain and range equivalence, that an orthogonal multiplication of type $[l,l,l]$ satisfies $l=1,2,4,$ or 8, and arises from the respective normed algebra product. This was generalized to the type $[l,m,m]$ by Radon~\cite{R}, who showed that such an orthogonal multiplication exists if and only if $l\leq\rho(m)$, the Hurwitz-Radon number. 
Adem~\cite{A},~\cite{A1} classified orthogonal multiplications of type $[l,m,m+1]$, while Gauchman and Toth~\cite{GT},~\cite{GT1} classified the type $[l,m,m+2]$.

There have been many extensive studies, in addition to the aforementioned, of the admissibility problem of orthogonal multiplications, i.e., of the existence of a given type $[m,n,p]$, in the literature (see~\cite{S} and the comprehensive references therein). 
The case when $m=n$ relates in a particularly interesting way to geometry in that the Hopf map $\Phi:(x,y)\mapsto (2x\circ y,|x|^2-|y|^2)$
from $S^{2m-1}$ to $S^p$ turns out to be harmonic, which has spurred many investigations~\cite{EL},~\cite{EL1},~\cite{HMX},~\cite{PT},~\cite{T},~\cite{To1},~\cite{To2},~\cite{TW},~\cite{WXZ} (and the references therein), to mention just a few.

Along a different direction, there appears to have been only sporadic studies of the moduli space, up to domain and range equivalence, of the types $[m,n,p]$ when one fixes $m\leq n$ and varies $p$ (it is understood that the orthogonal multiplications are full in each ${\mathbb R}^p$); note that necessarily $n\leq p\leq mn$. In this regard, we mention the work of Parker~\cite{P} when $m=n=3$, Toth~\cite{To} for an elegant, intrinsic structural framework in the case $m=n$, Guo~\cite{Gu} in the case $m=2$, and Toth~\cite{To1} in the comprehensive case $m=n=4$.

Lying between Parker's moduli $[3,3,p], p\leq 9,$ of dimension 3, and Toth's moduli $[4,4,p], p\leq 16$, of dimension 24, is the moduli $[3,4,p], p\leq 12$, which the title of the present paper addresses. Its subspace consisting of the type $[3,4,8]$ plays a vital role in the classification, initiated by Cartan~\cite{Car2},~\cite{Car4} of isoparametric hypersurfaces with four principal curvatures in the sphere~\cite{C}. 

Slightly extending~\cite{To}, we remark in Section~\ref{s2} that the moduli space of orthogonal multiplications of type $[m,n,p],$ where $m\leq n$ are fixed and $n\leq p\leq mn,$ up to range equivalence, can be identified with a compact convex set ${\mathcal C}$ in $\wedge^2 {\mathbb R}^m\otimes\wedge^2{\mathbb R}^n$ that is left fixed by the $SO(m)\otimes SO(n)$-actions. Therefore, up to domain and range equivalence, the moduli space is \,${\mathcal C}/SO(m)\otimes SO(n)$.

In the case of moduli space of type $[3,4,p], 4\leq p\leq 12,$ to be denoted by ${\mathcal M}_{(3,4,12)}$ henceforth, it is more convenient to represent the 18-dimensional
$\wedge^2{\mathbb R}^3\otimes\wedge^2{\mathbb R}^4$ via the isomorphism
$$
\wedge^2{\mathbb R}^3\otimes\wedge^2{\mathbb R}^4\simeq (\wedge^2{\mathbb R}^3\otimes\wedge^2{\mathbb R}^3)\oplus (\wedge^2{\mathbb R}^3\otimes\wedge^2{\mathbb R}^3)
$$
by the fact that over ${\mathbb R}^4$ a 2-form is the direct sum of a self-dual and an anti-self-dual form, which can then be identified with a 3-by-6 matrix with two 3-by-3 blocks that we refer to as the Parker matrix $C=\begin{pmatrix} A&B\end{pmatrix}$. 

For an orthogonal multiplication of type $[3,4,p],\, 4\leq p\leq 12,$ its associated Parker matrix can be so normalized, through domain equivalence, that $A$ is diagonal and $B$ is upper triangular, leaving us with 9 independent variables that give rise to a 9-dimensional coarse fundamental domain of ${\mathcal M}_{(3,4,12)}$ (see~\eqref{c1} and~\eqref{grand}). With this,
 in Section~\ref{348}, we show that the moduli of type $[3,4,8]$ has a rich structure carrying itself a grand moduli of dimension 5, denoted by ${\mathcal G}_{(3,4,8)}$, and an anomalous moduli ${\mathcal X}_{(3,4,8)}$ of dimension 3 that corresponds to the case when $A$ and $B$ in the normalized Parker matrix are both diagonal. 

As is the case for any coarse fundamental domain, there is a boundary identification to ``glue'' the domain into the actual moduli space. We point out in Section~\ref{348} that
these points are the ones, representing certain points in ${\mathcal G}_{(3,4,8)}$, for which at least one of the three off-diagonal entries of $B$ of the Parker matrix is zero, where each of them is identified with one or two other points by a subgroup of 
$S_3$, the permutation group on three letters, which in essence permutes the three off-diagonal entries of $B$ without destroying the diagonal-upper-triangular pattern of $A$ and $B$. 
As an application of this observation, in Section~\ref{reverse}, we consider the situation when the range equivalence is more rigidly restricted to a fixed decomposition ${\mathbb R}^8={\mathbb R}^4\oplus{\mathbb R}^4$, so that only $SO(4)\oplus SO(4)\subset SO(8)$ is allowed. This is pertinent to the classification of isoparametric hypersurfaces with four principal curvatures and multiplicity pair $(7,8)$ in $S^{31}$, to which a certain orthogonal multiplication of type $[3,4,8]$ is associated that respects the decomposition of ${\mathbb R}^8$ into two copies of ${\mathbb R}^4$ intrinsic to the isoparametric structure~\cite[Section 7]{C}. We prove in Section~\ref{reverse} that all such orthogonal multiplications of type $[3,4,8]$ come from the grand moduli ${\mathcal G}_{(3,4,8)}$, a property crucial for the classification of the isoparametric hypersurfaces in $S^{31}$ (see the ending paragraph of Section~\ref{reverse}).

We conclude in Section~\ref{finali} that ${\mathcal M}_{(3,4,12)}$ consists of the generic open set of dimension 9, and for each $0\leq n\leq 4,$ the generic moduli of type $[3,4,8+n]$ is of dimension $5+n$, which is achieved by studying the perturbation, via the normal exponential map, of the anomalous moduli ${\mathcal X}_{(3,4,8)}$.   
Inside ${\mathcal X}_{(3,4,8)}$ there sits a 1-dimensional
moduli of type $[3,4,7]$, which degenerates to the quaternion multiplication (of type $[3,4,4]$). In particular, there are no orthogonal multiplications of type $[3,4,5]$ and 
$[3,4,6]$; we remark that this is also a consequence of the aforementioned general results in~\cite{A},~\cite{A1},~\cite{GT},~\cite{GT1}.

The second author would like to express his gratitude to Professor Zizhou Tang for his guidance and encouragement. 
The first author would also like to thank him for the warm hospitality he received during his visit in Beijing in the Fall of 2016.

\section{The moduli space of orthogonal multiplications}\label{s2} In~\cite{To}, an elegant, intrinsic geometric picture is given to capture the moduli space of all orthogonal multiplications of type 
$[m,m,p]$ for fixed $m$ and for varying $p$, $m\leq p\leq m^2.$ In fact, this is completely general for the moduli space of orthogonal multiplications of type $[m,n,p]$, where $m\leq n$ are fixed and $n\leq p\leq mn$, which we briefly sketch.

A full orthogonal multiplication 
$F: {\mathbb R}^m \times {\mathbb R}^n \rightarrow {\mathbb R}^p, m\leq n,$ satisfies $n\leq p\leq mn$ for dimension reasons. Accordingly, we assume $p\leq mn$.  

Let $e_1, \cdots, e_m$ be an orthonormal basis for ${\mathbb R}^m$ and let $f_1,\cdots,f_n$ be an orthonormal basis for ${\mathbb R^n}$. We identify ${\mathbb R}^{mn}$ with
${\mathbb R}^m\otimes{\mathbb R}^n$. We set
$$
x\circ y=F(x,y).
$$ 
Following~\cite{P}, we define
$$
F_{ij,kl}=\langle e_i\circ f_j,e_k\circ f_l\rangle,
$$
where $\langle \cdot,\cdot\rangle$ denotes the standard Euclidean inner product. We have the properties
\begin{equation}\label{eq}
\aligned
&F_{ij,kl}=F_{kl,ij},\quad \forall i,j,k,l,\\
&F_{ij,il}=F_{li,ji}=0, \quad j\neq l,\\
&F_{ij,kl}=-F_{il,kj},\quad i\neq k, \\
& F_{ij,kl}=-F_{kj,il},\quad j\neq l,
\endaligned
\end{equation}
from which the quantity 
\begin{equation}\label{bracket}
[ij,kl]:=F_{ik,jl}
\end{equation}
satisfies
$$
[ij,kl]=-[ji,kl],\quad [ij,kl]=-[ij,lk].
$$

Intrinsically, this means that if we define
\begin{equation}\label{eqq}
\aligned
&g:=\sum_{(ij,kl)} F_{ij,kl}\;(e_i\otimes f_j)\odot (e_k\otimes f_l),\\
&\iota:=\sum_{(ij,ij)}\;(e_i\otimes f_j)\odot (e_i\otimes f_j),
\endaligned
\end{equation}
in $S^2({\mathbb R}^m\otimes {\mathbb R}^n)$, set
$$
W:=\text{span}((a\otimes b)\odot (a\otimes b))\subset S^2({\mathbb R}^m\otimes {\mathbb R}^n),
$$
and let $W^{\perp}$ be the space perpendicular to $W$ in $S^2({\mathbb R}^m\otimes {\mathbb R}^n)$. Then by~\eqref{eq},
$$
g-\iota\in W^{\perp}.
$$
Moreover, if we let
\begin{equation}
g-\iota := c=\sum_{(ij,kl)} c_{ij,kl}\;(e_i\otimes f_j)\odot (e_k\otimes f_l),
\end{equation}
then $c_{ij,kl}$ satisfies
$$
c_{ij,kl}=[ik,jl].
$$
Conversely, for any element $c$ in $W^\perp$, the symmetric tensor 
$$
g:=\iota+c:=\sum_{(ij,kl)} F_{ij,kl}\;(e_i\otimes f_j)\odot (e_k\otimes f_l)
$$
has the same properties as in~\eqref{eq} so that 
$$
[ik,jl]:=F_{ij,kl}
$$ 
satisfies
$$
[ij,kl]=-[ji,kl],\quad [ij,kl]=-[ij,lk].
$$
As a consequence, 
$$
W^\perp\simeq \wedge^2{\mathbb R}^m\otimes \wedge^2{\mathbb R}^n.
$$ 
Meanwhile, each $c\in W^\perp$ induces a traceless symmetric endomorphism
$$
c^*: {\mathbb R}^m\otimes {\mathbb R}^n\rightarrow {\mathbb R}^m\otimes {\mathbb R}^n\quad c^*:e_i\otimes f_j\mapsto \sum_{kl}c_{ij,kl}\;e_k\otimes f_l,
$$
In particular, when $c$ is induced by an orthogonal multiplication $F$ as in~\eqref{eqq}, we have the extra property that for 
$$
X=\sum_{ij}X_{ij}\;e_i\otimes f_j\in {\mathbb R}^m\otimes {\mathbb R}^n,
$$
there follows
$$
\langle (id+c^*)(X),X\rangle=\sum_{ij,kl}F_{ij,kl}X_{ij}X_{kl}=\langle\sum_{ij}X_{ij}\;e_i\circ f_j,\sum_{ij}X_{ij}\;e_i\circ f_j\rangle\geq 0,
$$
so that $\iota+c$ regarded as a symmetric endomorphism is semi-positive-definite. To recover $x\circ y$, by the universal property of tensor product, there is a unique 
endomorphism
$D:{\mathbb R}^m\otimes{\mathbb R}^n\rightarrow {\mathbb R}^p\subset {\mathbb R}^m\otimes{\mathbb R}^n$ such that 
$x\circ y=D(x\otimes y)$. We have
$$
F_{ij,kl}=\langle e_i\circ f_j,e_k\circ f_l\rangle=\langle D^{tr}D(e_i\otimes f_j),e_k\otimes f_l\rangle.
$$
That is, by~\eqref{eqq},
$$
D^{tr}D=g^*,
$$
where $g^*$ is the associated endormorphism of $g$.

Conversely, given a semi-positive-definite symmetric $g=\iota+c$, 
let $g^*=id+c^*$ be the associated endomorphism. Let $E_0$ be the 0-eigen space of $g^*$, and let $E_0^\perp$ be its orthogonal complement in ${\mathbb R}^m\otimes{\mathbb R}^n$. Let $n_1,\cdots,n_p$ be eigenvectors of $E_0^\perp$ with eigenvalues $\lambda_1,\cdots,\lambda_p>0$. 
For any $p$ orthonormal elements $v_1,\cdots,v_p$ in ${\mathbb R}^m\otimes{\mathbb R}^n$ 
define
$$
D_v:{\mathbb R}^m\otimes{\mathbb R}^n\rightarrow {\mathbb R}^m\otimes{\mathbb R}^n,\quad n_i\mapsto \sqrt{\lambda_i} v_i,\quad E_0\mapsto 0.
$$
Then $D_v$ satisfies 
\begin{equation}\label{D}
D_v^{tr}D_v=g^*,
\end{equation}
and, moreover, this exhausts all possible $D_v$ satisfying $D_v^{tr}D_v=g^{*}$. We define 
$$
e_i\circ_v f_j:=D_v(e_i\otimes f_j)\in \text{span}(v_1,\cdots,v_p)\simeq{\mathbb R}^p,
$$
from which it is checked by~\eqref{D} that for $|x|=|y|=1$
$$
|x\circ_v y|^2 = \sum_{(ij,kl)} x_i y_jx_ky_l\;\langle D_v(e_i\otimes f_j),D_v(e_k\otimes f_l)\rangle
= \sum_{(ij,kl)} F_{ij,kl}\;x_i y_jx_ky_l=1,
$$
where we set $x=e_1$ and $y=se_1+te_2, s^2+t^2=1,$ for a convenient calculation by~\eqref{eq}. That is, $F_v(x,y)=x\circ_v y$ so defined is an orthogonal multiplication.
Note that $D_v$ differs from $D_{w}$ only by an orthonormal basis change in ${\mathbb R}^m\otimes{\mathbb R}^n$, so that they induce the same orthogonal multiplication up to range equivalence.

As a result, up to range equivalence, the set of orthogonal multiplications is a compact convex subset ${\mathcal C}$ (call it the {\em preliminary moduli space} in the following) of $W^\perp\simeq \wedge^2{\mathbb R}^m\otimes \wedge^2{\mathbb R}^n$, consisting of $c$ for which $\iota+c$ is semi-positive-definite. Its interior consists of
full orthogonal multiplications, and the non-full orthogonal multiplications lie on the boundary.

Since the domain equivalence $SO(m)\otimes SO(n)$ acts on $W^\perp$ and fixes\, ${\mathcal C}$, the moduli space of orthogonal multiplications of type $[m,n,p]$, up to domain and range equivalence, where $m\leq n$ are fixed and $n\leq p\leq mn,$ is\, ${\mathcal C}/SO(m)\otimes SO(n)$. 

\section{Type $[2,4,8]$}\label{s3} We recount the result in~\cite{Gu}, specialized to our situation with a slightly different approach, to gain motivation for the subsequent development. Since the preliminary moduli space ${\mathcal C}\subset \wedge{\mathbb R}^2\otimes \wedge^2{\mathbb R}^4=\wedge^2{\mathbb R}^4$. This case simply says if we choose the basis elements
$u_5,\cdots,u_8$ to be $e_2\circ f_1,\cdots,e_2\circ f_4$, and  complete the basis to $u_1,\cdots, u_4, u_5,\cdots u_8$, then the two Hurwitz matrices defined by
$$
F_a=\text{the matrix whose}\;b\text{th row is}\; e_a\circ f_b,\quad 1\leq a\leq 2,
$$
are 
$$
F_1=\begin{pmatrix} c_1&w_1\end{pmatrix},\quad F_2=\begin{pmatrix}0&Id\end{pmatrix},
$$
where the 4-by-4 $w_1$ is skew-symmetric and gives the above $\wedge^2{\mathbb R}^4$. Knowing that the connected ${\mathcal C}$ is six-dimensional, we can immediately 
find its structure:
$$
\aligned
&F_1=\begin{pmatrix} UDU^{tr}&UKU^{tr}\end{pmatrix},\quad D:=\begin{pmatrix}\sigma_1&0&0&0\\0&\sigma_2&0&0\\0&0&\sigma_2&0\\0&0&0&\sigma_1\end{pmatrix},\quad
K:=\begin{pmatrix} 0&0&0&-\mu\\0&0&-\nu&0\\0&\nu&0&0\\\mu&0&0&0\end{pmatrix},\\
&UU^{trr}=Id,\quad \sigma_1^2+\mu^2=1,\quad \sigma_2^2+\nu^2=1.
\endaligned
$$
where the adjoint orbit of $K$ by $U$ is only 4-dimensional. 

It is easily checked that $F_1$ so given and $F_2:=\begin{pmatrix}0&I\end{pmatrix}$ form a Hurwitz system of dimension 6. Conversely, every $F_1$ can be brought to this form by appropriate coordinate changes. This is because the skew-symmetry of $w_1$ allows us to find an orthogonal matrix $U$ such that $w_1=UKU^{tr}$, so that by performing a coordinate change on $span(f_1,\cdots,f_4)$
and on $span(u_5,\cdots,u_8)$, leaving $F_2$ fixed, we may assume $F_1=\begin{pmatrix}U^{tr}c_1&K\end{pmatrix}$. Since 
$$
(U^{tr}c_1)(U^{tr}c_1)^{tr}+KK^{tr}=Id,
$$
it follows that the first term in the sum is diagonal, or rather, the rows of $U^{tr}c_1$ are mutually orthogonal, so that we may find an orthogonal matrix $W$ such that $U^{tr}c_1=DW$. Hence, after a coordinate change on $span(u_1,\cdots,u_4)$, fixing $F_2$, we may assume 
\begin{equation}\label{222}
F_1=\begin{pmatrix}D&K\end{pmatrix},\quad F_2=\begin{pmatrix}0&Id\end{pmatrix},
\end{equation} 
which constitute the moduli space of type $[2,4,p], p\leq 8$,
through the action $SO(2)\otimes SO(4)$ on $\wedge^2{\mathbb R}^2\otimes\wedge^2{\mathbb R}^4$, where the left action is trivial while the right one is the adjoint action.

\section{Type $[3,4,8]$}\label{348}
\subsection{The setup}\label{4.1} 
Given an orthogonal multiplication $F$ of type $[3,4,8]$, its Hurwitz matrices $F_a,1\leq a\leq 3,$ defined by
$$
F_a:=\text{the matrix whose}\; b\text{th row is}\;e_a\circ f_b
$$
with $x\circ y:=F(x,y)$, satisfies the Hurwitz condition
\begin{equation}\label{Hz}
F_aF_b^{tr}+F_bF_a^{tr}=2\delta_{ab}\, Id,\quad 1\leq a,b\leq 3.
\end{equation}
We can choose the ``anchor'' matrix to be 
$$
F_3=\begin{pmatrix}0&Id\end{pmatrix}
$$
as in the type $[2,4,8]$. We seek to express $F_1$ and $F_2$ in the most convenient form for analysis.~\eqref{222} will be our guide to formulate $F_1$. To achieve the goal, recall~\eqref{bracket} and consider 
the 3-by-6 matrix
\begin{equation}\label{P}
C:=\begin{pmatrix}[ij,kl]\end{pmatrix},
\end{equation}
referred to henceforth as the Parker matrix~\cite{P}, which defines a linear map 
\begin{equation}\label{-1}
C:\wedge^2({\mathbb R}^4)\rightarrow\wedge^2({\mathbb R}^3).
\end{equation}
We identify the basis of the target space by
$$
e_1\wedge e_2,\quad e_3\wedge e_1,\quad e_3\wedge e_2,
$$
and identify the domain space by
$$
\wedge^{2}({\mathbb R}^4)=\wedge^2({\mathbb R}^3)\oplus\wedge^2({\mathbb R}^3)
$$
where the right hand side consists of the space of self-dual forms and of anti-self-dual forms, where the former is spanned by
$$
f_1\wedge f_2+f_3\wedge f_4,\quad f_1\wedge f_3+f_4\wedge f_2,\quad f_1\wedge f_4+f_2\wedge f_3,
$$
and the latter by
$$
f_1\wedge f_2-f_3\wedge f_4,\quad f_1\wedge f_3-f_4\wedge f_2,\quad f_1\wedge f_4-f_2\wedge f_3.
$$
As a consequence
\begin{equation}\label{AB}
C =\begin{pmatrix} A&B\end{pmatrix},
\end{equation}
where $A$ and $B$ are of size 3-by-3. 

As in~\cite{P}, $A$ and $B$ can be written in the form
$$
A=U^{tr}D_1W_1U,\quad B=V^{tr}D_2W_2V,
$$
where $U,V,W_1,W_2$ are orthogonal and $D_1 ,D_2$ are diagonal, so that by applying orthogonal changes to row and column spaces, we may assume
\begin{equation}\label{0}
C=\begin{pmatrix}D&T\end{pmatrix}
\end{equation}
with D diagonal and T upper triangular. This means that we may assume
\begin{equation}\label{-10}
\aligned
&[31,12\pm 34]= [32,12\pm34]=[32,13\pm 42]=0,\\
&[12,13+42]=[12,14+23]=[31,14+23]=0;
\endaligned
\end{equation}
in other words, we now have the identities
\begin{equation}\label{1}
\aligned
&F_{31,12}=F_{33,14}=0,\quad F_{31,22}=F_{33,24}=0,\quad F_{31,23}=F_{34,22}=0,\\
&F_{11,23}+F_{14,22}=F_{11,24}+F_{12,23}=F_{31,14}+F_{32,13}=0.
\endaligned
\end{equation}
With this choice of basis, we can now pick the orthonormal set 
$$
u_5:=e_3\circ f_1,\;\; u_6:=e_3\circ f_2,\;\; u_7:=e_3\circ f_3, \;\;\;u_8:=e_3\circ f_4,
$$ 
relative to which we have
\begin{equation}\label{anchor}
F_3=\begin{pmatrix}c_3&w_3\end{pmatrix}=\begin{pmatrix}0&Id\end{pmatrix}.
\end{equation}
It is then calculated that the orthogonal projection of 
$$
v_1:=e_2\circ f_1,\;\;v_2:=e_2\circ f_2,\;\;v_3:=e_2\circ f_3, \;\;v_4:=e_2\circ f_4
$$ 
onto $span(u_5,\cdots,u_8)$, denoted $v_1^\perp,\cdots,v_4^\perp$, are, respectively,
$$
v_1^\perp=F_{21,34}\;u_8,\quad v_2^\perp=F_{22,33}\;u_7,\quad v_3^\perp=F_{23,32}\;u_6,\quad v_4^\perp=F_{24,31}\;u_5,
$$
which are mutually orthogonal with 
$$
|v_1|=|v_4|,\quad |v_2|=|v_3|,
$$
so that we may complete the orthonormal basis $u_1,\cdots,u_8$ by setting
$$
u_i:=(e_2\circ f_i -v_i^\perp)/(\text{length}),\quad 1\leq i\leq 4. 
$$
Consequently, relative to $u_1,\cdots,u_8$ we have
\begin{equation}\label{c1}
\aligned
&F_1=\begin{pmatrix}c_1&w_1\end{pmatrix},\quad c_1=\begin{pmatrix}\sigma_1&0&0&0\\0&\sigma_2&0&0\\0&0&\sigma_2&0\\0&0&0&\sigma_1\end{pmatrix},\quad w_1=\begin{pmatrix} 0&0&0&-\mu\\0&0&-\nu&0\\0&\nu&0&0\\\mu&0&0&0\end{pmatrix},\\
&\mu:=F_{24,31},\quad \nu:=F_{23,32},\quad \sigma_1:=\sqrt{1-\mu^2},\quad \sigma_2:=\sqrt{1-\nu^2}.
\endaligned
\end{equation}
Now, $e_1\circ f_1,\cdots,e_1\circ f_4$ can be expanded in terms of $u_1,\cdots,u_8$ to yield $F_2=\begin{pmatrix}c_2&w_2\end{pmatrix}$, where
\begin{equation}\label{grand}
\aligned
&c_2=\begin{pmatrix}-\beta\mu/\sigma_1&(F_{11,22}-\alpha\nu)/\sigma_2&F_{11,23}/\sigma_2&F_{11,24}/\sigma_1\\(F_{1221}-\mu\gamma)/\sigma_1&\beta\nu/\sigma_2&F_{12,23}/\sigma_2&F_{12,24}/\sigma_1\\F_{13,21}/\sigma_1&F_{13,22}/\sigma_2&\beta\nu/\sigma_2&(F_{13,24}-\alpha\mu)/\sigma_1\\F_{14,21}/\sigma_1&F_{14,22}/\sigma_2&(F_{14,23}-\gamma\nu)/\sigma_2&-\beta\mu/\sigma_1\end{pmatrix},\\
&w_2=\begin{pmatrix}0&0&-\alpha&-\beta\\0&0&\beta&-\gamma\\\alpha&-\beta&0&0\\\beta&\gamma&0&0\end{pmatrix},\quad\alpha:=F_{31,13},\quad \beta=F_{31,14},\quad \gamma:=F_{32,14}.
\endaligned
\end{equation}
Note that when either $\sigma_1=0$ or $\sigma_2=0$, the corresponding columns for $c_2$ give, respectively,
\begin{equation}\label{degenerate}
\aligned
&\beta=0,\quad F_{11,22}=-\gamma\mu,\quad F_{14,23}=-\alpha\mu,\quad F_{14,21}=F_{14,22}=0,\quad \text{if}\;\; \sigma_1=0,\\
&\beta=0,\quad F_{11,22}=\alpha\nu,\quad F_{14,23}=\gamma\nu,\quad F_{14,21}=F_{14,22}=0,\quad\text{if}\;\; \sigma_2=0.
\endaligned
\end{equation}
This belongs to the most degenerate case we will consider later. As a corollary, we have

\begin{lemma}\label{l3} If $\beta\neq 0$, then $|\mu|,|\nu|<1$. That is, $\sigma_1 \sigma_2\neq 0$.
\end{lemma}
It is then legitimate to perform operations with the entries in $c_2$ when $\beta\neq 0$.

With the normalization of the Parker matrix given in~\eqref{-10}, or equivalently, in~\eqref{1}, we have chosen a coarse fundamental domain for the moduli space of orthogonal multiplications of type $[3,4,8]$ in~\eqref{c1} and~\eqref{grand}, up to domain and range equivalence. There remains a boundary identification of the coarse fundamental domain to be addressed in Section~\ref{sb}.
 
 Note that the Hurwitz condition~\eqref{Hz} is now 
 $$
 c_ac_b^{tr}+c_bc_a^{tr} + w_aw_b^{tr} + w_bw_a^{tr}=2\delta_{ab}\, Id,\quad 1\leq a, b\leq 3.
 $$

\subsection{ The generic case when $\beta\neq 0$}

 \begin{lemma}\label{l1} Assume $\beta\neq 0$. If $F_{14,22}\neq 0$,  then 
$$
\alpha=\gamma,\quad \mu=-\nu, \quad F_{11,22}=F_{14,23}.
$$
\end{lemma}

\begin{proof} Since the rows of $F_2$ are mutually orthogonal, of which the first and the fourth, and, respectively, the second and the third, give
\begin{equation}\label{ga}
F_{14,22}\,((F_{11,22}-F_{14,23})+\nu(\gamma-\alpha))=0,\quad F_{14,22}\,((F_{11,22}-F_{14,23})-\mu(\gamma-\alpha))=0,
\end{equation}
from which there results
$$
(\mu+\nu)(\gamma-\alpha)=0,
$$
so that either $\alpha=\gamma$ or $\mu=-\nu$. 

If $\mu=-\nu$, then $\sigma_1=\sigma_2:=\sigma$ and the first and the second rows of $F_2$ give
\begin{equation}\label{referral}
0=\beta F_{11,22}(\mu+\nu)+\beta(\mu^2+\sigma^2)(\gamma-\alpha)=\beta(\gamma-\alpha),\quad \mu^2+\sigma^2=1,
\end{equation}
we obtain $\alpha=\gamma$.

On the other hand, if $\alpha=\gamma,$
then~\eqref{ga} gives
$F_{11,22}=F_{14,23}.$
Meanwhile, the rows of $c_2$ are now mutually orthogonal and of equal length, which implies that the columns of $c_2$ are mutually orthogonal and of the same length, of
which the first and the second give
$$
\beta F_{11,22}(\mu+\nu)=0.
$$
That is, $\mu=-\nu$
when $F_{11,22}=F_{14,23}\neq 0$. 

Suppose now 
$F_{11,22}=F_{14,23}=0.$
Since the third and fourth rows of $c_2$ are orthogonal, we obtain
$$
(F_{14,21}F_{14,22}+\alpha\beta)\;\frac{\mu^2-\nu^2}{\sigma_1^2\sigma_2^2}=0,
$$
so that 
\begin{equation}\label{albt}
F_{14,21}F_{14,22}=-\alpha\beta
\end{equation}
if 
$\mu^2\neq \nu^2.$
Similarly, the second and the fourth rows of $c_2$ give 
$$
\alpha F_{14,21}(\frac{\mu}{\sigma_1^2}+\frac{\nu}{\sigma_2^2})=\beta F_{14,22}(\frac{\mu}{\sigma_1^2}+\frac{\nu}{\sigma_2^2}),
$$
where
$$
\frac{\mu}{\sigma_1^2}+\frac{\nu}{\sigma_2^2}=\frac{(u+v)(1-uv)}{\sigma_1^2\sigma_2^2}\neq 0.
$$
Therefore, we derive
$$
\alpha F_{14,21}=\beta F_{14,22}, 
$$
which, upon substituting~\eqref{albt}, results in
$$
(\beta F_{14,22})^2=\alpha\beta F_{14,21}F_{14,22}=-(\alpha\beta)^2.
$$
It follows that
$$
\alpha\beta=0,\quad \beta F_{14,22}=0,
$$
so that
$F_{14,22}=0,$
contradicting $F_{14,22}\neq 0$, and thus
$\mu^2=\nu^2.$
So, after a coordinate sign change, we may assume 
$\mu=-\nu,$ 
from which~\eqref{referral} results in
$\alpha=\gamma.$
\end{proof}

\begin{lemma} Assume $\beta\neq 0$. If $F_{14,22}=0$ and $F_{14,21}\neq 0$,  then 
$$
\alpha=\gamma,\quad \mu=-\nu, \quad F_{11,22}=F_{14,23}.
$$
\end{lemma}

\begin{proof} With $F_{14,22}=0$, the first and the second rows, and respectively the third and the fourth rows, give, after cancelling $\beta$,
\begin{equation}\label{twins}
\aligned
&F_{11,22}\,(\frac{\mu}{\sigma_1^2}+\frac{\nu}{\sigma_2^2})+(\frac{\mu^2\gamma}{\sigma_1^2}-\frac{\nu^2\alpha}{\sigma_2^2})+ (\gamma-\alpha)=0,\\
&F_{14,23}\,(\frac{\mu}{\sigma_1^2}+\frac{\nu}{\sigma_2^2})+(\frac{\mu^2\alpha}{\sigma_1^2}-\frac{\nu^2\gamma}{\sigma_2^2})+ (\alpha-\gamma)=0.
\endaligned
\end{equation}
Meanwhile, the first and the third rows, and, respectively, the second and the fourth rows, give, after cancelling $F_{14,21}$,
\begin{equation}\label{TWINS}
\aligned
&\frac{F_{11,22}}{\sigma_2^2}-\frac{F_{14,23}}{\sigma_1^2}=\alpha\,(\frac{\mu}{\sigma_1^2}+\frac{\nu}{\sigma_2^2}),\\
&\frac{F_{11,22}}{\sigma_1^2}-\frac{F_{14,23}}{\sigma_2^2}=-\gamma\,(\frac{\mu}{\sigma_1^2}+\frac{\nu}{\sigma_2^2}).
\endaligned
\end{equation}
We substitute~\eqref{twins} into~\eqref{TWINS} to come up with, respectively,
$$
\alpha\,(2-(\mu+\nu)^2)=2\gamma,\quad \gamma\,(2-(\mu+\nu)^2)=2\alpha
$$
using $\mu^2+\sigma_1^2=\nu^2+\sigma_2^2=1$, from which we conclude that 
$$
\text{either}\quad\alpha=\gamma=0,\quad \text{or}\quad 2-(\mu+\nu)^2=\pm 2.
$$
However, since $|\mu|+|\nu|<2$ by Lemma~\ref{l3}, it is impossible that $2-(\mu+\nu)^2=-2$. There follows 
$$
\text{either}\quad\alpha=\gamma=0,\quad \text{or}\quad \mu+\nu=0.
$$
In the former case, $c_2$ is orthogonal up to a constant, so that its first and second columns are of some length $l$, which amounts to
$$
\beta^2\mu^2+F_{11,22}^2+F_{14,21}^2=l^2\sigma_1^2,\quad \beta^2\nu^2+F_{11,22}^2+F_{14,21}^2=l^2\sigma_2^2,
$$
so that 
$$
\beta^2(\mu^2-\nu^2)=l^2(\sigma_1^2-\sigma_2^2)=-l^2(\mu^2-\nu^2), \quad \text{or}\quad (\beta^2+l^2)(\mu^2-\nu^2)=0.
$$
That is, with $\beta\neq 0$, we have
\begin{equation}\label{mn}
\mu^2=\nu^2,
\end{equation}
and we may assume
$\mu=-\nu$
by a coordinate sign change. Now that $\mu=-\nu$ in either case,~\eqref{twins} then gives
$\alpha=\gamma,$
and~\eqref{TWINS} gives
$F_{11,22}=F_{14,23}$
since now $\sigma_1=\sigma_2$.
\end{proof}

\begin{lemma} Assume $\beta\neq 0$. If $F_{14,21}=F_{14,22}=0$,  then 
$$
\alpha=\gamma,\quad \mu=-\nu.
$$Moreover, either $F_{14,23}=F_{11,22}$, or $F_{14,23}=-F_{11,22}-2\alpha\mu$.
\end{lemma}

\begin{proof} First note that~\eqref{twins} continues to hold, so that we obtain
\begin{equation}\label{4eqq}
\aligned
&(F_{11,22}-\alpha\nu)(\frac{\mu}{\sigma_1^2}+\frac{\nu}{\sigma_2^2})=\frac{-\mu(\mu\gamma+\alpha\nu)}{\sigma_1^2}-(\gamma-\alpha),\\
&(F_{11,22}+\mu\gamma)(\frac{\mu}{\sigma_1^2}+\frac{\nu}{\sigma_2^2})=\frac{\nu(\mu\gamma+\alpha\nu)}{\sigma_2^2}-(\gamma-\alpha),\\
&(F_{14,23}+\alpha\mu)(\frac{\mu}{\sigma_1^2}+\frac{\nu}{\sigma_2^2})=\frac{\nu(\mu\alpha+\gamma\nu)}{\sigma_2^2}-(\alpha-\gamma),\\
&(F_{14,23}-\gamma\nu)(\frac{\mu}{\sigma_1^2}+\frac{\nu}{\sigma_2^2})=\frac{-\mu(\mu\alpha+\gamma\nu)}{\sigma_1^2}-(\alpha-\gamma).
\endaligned
\end{equation}
Meanwhile, the four rows of $F_2$ being of length 1 translates to
\begin{equation}\label{4eq}
\aligned
&\frac{(F_{11,22}-\alpha\nu)^2}{\sigma_2^2}=1-\alpha^2-\frac{\beta^2}{\sigma_1^2},\\
&\frac{(F_{11,22}+\gamma\mu)^2}{\sigma_1^2}=1-\gamma^2-\frac{\beta^2}{\sigma_2^2},\\
&\frac{(F_{14,23}+\alpha\mu)^2}{\sigma_1^2}=1-\alpha^2-\frac{\beta^2}{\sigma_2^2},\\
&\frac{(F_{14,23}-\gamma\nu)^2}{\sigma_2^2}=1-\gamma^2-\frac{\beta^2}{\sigma_1^2},
\endaligned
\end{equation}
from which we derive
\begin{equation}\label{cross}
\frac{(F_{11,22}+\gamma\mu)^2}{\sigma_1^2}-\frac{(F_{14,23}+\alpha\mu)^2}{\sigma_1^2}
=\frac{(F_{14,23}-\gamma\nu)^2}{\sigma_2^2}-\frac{(F_{11,22}-\alpha\nu)^2}{\sigma_2^2}.
\end{equation}
Inserting~\eqref{4eqq} into~\eqref{cross} and cancelling $\sigma_1^2\sigma_2^2$ yields
$$
\aligned
&\frac{(\,\mu\alpha(\mu+\nu)+(\gamma-\alpha)\,)^2}{\sigma_1^2}-\frac{(\,-\mu\gamma(\mu+\nu)+(\gamma-\alpha)\,)^2}{\sigma_1^2}\\
&=\frac{(\alpha\nu(\mu+\nu)+(\gamma-\alpha)\,)^2}{\sigma_2^2}-\frac{(\,\nu\gamma(\mu+\nu)-(\gamma-\alpha)\,)^2}{\sigma_2^2}.
\endaligned
$$
This further simplifies to
$$
(\mu+\nu)(\alpha^2-\gamma^2)(\mu^2(\mu+\nu)-2\mu)\sigma_2^2=(\mu+\nu)(\alpha^2-\gamma^2)(\nu^2(\mu+\nu)-2\nu)\sigma_1^2,
$$
so that it finally arrives at
\begin{equation}\label{=0}
(\mu^2-\nu^2)(\alpha^2-\gamma^2)(\mu^2-\nu^2-2)=0.
\end{equation}
As a consequence,~\eqref{=0} gives, since $\mu^2-\nu^2-2\neq 0$ by Lemma~\ref{l3}, that
$$
\text{either}\quad\alpha^2=\gamma^2,\quad \text{or,}\quad \mu^2=\nu^2.
$$

If $\mu^2=\nu^2$, then we may assume 
$\mu=-\nu$
by a coordinate sign change; with $\sigma=\sigma_1=\sigma_2$ now, this implies by~\eqref{twins}
\begin{equation}\label{fin}
(\frac{\mu^2}{\sigma^2}+1)(\gamma-\alpha)=0,
\end{equation}
so that 
$\alpha=\gamma.$

If $\alpha^2=\gamma^2$, then $c_2$ is orthogonal up to a constant. 
%%we may assume
%%$$
%%\alpha=\gamma
%%$$ 
%%by a coordinate sign change. 
%%Suppose $\mu^2\neq \nu^2$. Then on the one hand,~\eqref{twins} gives, after cancelling $\mu+\nu$,
%%\begin{equation}\label{>0}
%%F_{11,22}=\frac{-\alpha(\mu-\nu)}{1-\mu\nu},\quad 1-\mu\nu>0.
%%\end{equation}
%%Now the column vectors of $c_2$ must be mutually orthogonal because its rows are orthogonal and of equal length. The first and the second columns give
%%$$
%%F_{11,22}\,(\mu+\nu)=0,
%%$$ 
%%so that $F_{11,22}=0.$ But then~\eqref{>0} implies $\alpha=0$. 
Hence, the same argument leading to~\eqref{mn} results in $\mu^2=\nu^2$.
By a coordinate sign change, we may assume
$\mu=-\nu,$
so that once more~\eqref{fin} gives
$\alpha=\gamma.$

Now that $\alpha=\gamma$ and $\mu=-\nu$,~\eqref{twins} becomes void. The only essential equations left are the first and the fourth ones in~\eqref{4eq}, which read

$$
(F_{11,22}+\alpha\mu)^2=(F_{14,23}+\alpha\mu)^2=(1-\mu^2)(1-\alpha^2)-\beta^2.
$$
Therefore,  
$$
\text{either}\;\;\;F_{11,22}+F_{14,23},\;\;\;\text{or}\;\;\; F_{14,23}=-F_{11,22}-2\alpha\mu.
$$
In both cases, $\beta$ satisfies the constraint
$$
\beta^2=(1-\mu^2)(1-\alpha^2)-(F_{11,22}+\alpha\mu)^2.
$$
\end{proof}

We summarize the above analysis in the following proposition.

\begin{proposition}\label{prop} Assume $\beta\neq 0$. Then
\begin{equation}\label{mod1}
\alpha=\gamma,\quad \mu=-\nu.
\end{equation}
Moreover, if either $F_{14,21}$ or $F_{14,22}$ is nonzero, then we have the further constraint
\begin{equation}\label{mod2}
F_{11,22}=F_{14,23}.
\end{equation}
The moduli of such orthogonal multiplications of type $[3,4,8]$ depends on the five variables $\alpha,\mu,F_{14,21},F_{14,22},F_{11,22}$, while $\beta$
is linked with the five variables by
\begin{equation}\label{14}
\beta^2=(1-\mu^2)\;(\,(1-\alpha^2)-(F_{11,22}+\alpha\mu)^2-F_{14,21}^2-F_{14,22}^2\,).
\end{equation}

On the other hand, if $F_{14,21}=F_{14,22}=0$, then 
\begin{equation}\label{qq}
\text{either}\;\;\;F_{11,22}=F_{14,23},\;\;\;\text{or}\;\;\; F_{14,23}=-F_{11,22}-2\alpha\mu.
\end{equation}
Both cases satisfy~\eqref{14} when we set $F_{14,21}=F_{14,22}=0$.
\end{proposition}

\begin{definition} We will refer to the moduli represented by orthogonal multiplications of type $[3,4,8]$, where $\alpha=\gamma$ and $\mu=-\nu$, as the {\bf grand} moduli ${\mathcal G}_{(3,4,8)}$. 
\end{definition}

Note that the condition $\alpha=-\beta$ and $\mu=\nu$ are equivalent to the one given in the definition up to a coordinate sign change.

In conclusion, all orthogonal multiplications with $\beta\neq 0$ are in the grand moduli ${\mathcal G}_{(3,4,8)}$, up to domain and range equivalence.

\subsection{$\beta=0$, the generic case when $F_{14,21}^2+F_{14,22}^2\neq 0$}\label{sb} 

The element of $SO(4)$ that fixes $f_1$,  maps $f_2$ to its negative, and interchanges $f_3$ and $f_4$, so that it transforms the self-dual forms to self-dual forms as follows,
$$
\aligned
&f_1\wedge f_2+f_3\wedge f_4\mapsto -(f_1\wedge f_2+f_3\wedge f_4),\\
& f_1\wedge f_3+f_4\wedge f_2 \mapsto f_1\wedge f_4+f_2\wedge f_3,\\
&f_1\wedge f_4+f_2\wedge f_3\mapsto f_1\wedge f_3+f_4\wedge f_2,
\endaligned
$$
and meanwhile interchanges anti-self-dual forms likewise. 

It follows that when we replace $f_2$ by its negative, and interchange the pair $f_3$ and $f_4$ and the pair $e_1$ and $e_2$, we will preserve the same type of
decomposition except for a possible sign change and a permutation of the last two diagonal entries of $D_1$ and $D_2$. Note that under the transformation
\begin{equation}\label{dual}
\aligned
&e_1\longleftrightarrow e_2,\quad \quad e_3\rightarrow e_3,\\
&f_3\longleftrightarrow f_4 ,\quad f_1\rightarrow f_1,\quad f_2\rightarrow -f_2,
\endaligned
\end{equation}
the following data are exchanged:

\begin{equation}\label{trans}
\aligned
&\alpha\longleftrightarrow\mu,\quad \gamma\longleftrightarrow -\nu,\quad F_{11,22}\longleftrightarrow F_{11,22},\quad F_{14,23}\longleftrightarrow F_{14,23},\\
&F_{1421}\longleftrightarrow F_{23,11}=-F_{14,22}
\endaligned
\end{equation}
Therefore, whenever an identity involving only these  quantities hold, the transformed identity via~\eqref{trans} must hold as well. 

The swapping produces two different representatives of the same moduli point.

\begin{proposition}\label{ll} Assume $\beta=0$. If  $F_{14,21}F_{14,22}\neq 0$, then we have
$$
\alpha=\gamma, \quad \mu=-\nu,\quad F_{11,22}=F_{14,23}.
$$
It is part of the grand moduli ${\mathcal G}_{(3,4,8)}$ when we set $\beta=0$.

If $F_{14,21}\neq 0$ and $F_{14,22}=0$, then 
$$
\mu=-\nu,\quad F_{11,22}=F_{14,23}.
$$ 
 Furthermore, if $\alpha=\gamma$, then it is part of the grand moduli~\eqref{14} in which we set $\beta=F_{14,22}=0$, depending on three parameters. On the other hand,
%%The anomalous case $\alpha\neq \gamma$ is governed by~\eqref{F} below.
if $\alpha\neq \gamma$, then 
these orthogonal multiplications are governed by~\eqref{useful} and~\eqref{yum} below and depends on the three parameters $\alpha,\mu,$ and $\gamma$.

The case when $F_{14,21}=0$ and $F_{14,22}\neq 0$ is equivalent to the preceding one via~\eqref{trans}.

In fact, there is a one-to-one correspondence between the case in which $\beta\neq 0$ and $F_{14,21}=F_{1422}=0$ in Proposition~{\rm \ref{prop}}, and the above case in which $\beta=0$ and $F_{14,21}\neq 0$ and $F_{14,22}=0$ {\rm (}or $F_{14,21}=0$ and $F_{14,22}\neq 0${\rm )}. 
\end{proposition}

\begin{proof} By~\eqref{degenerate}, we have $\sigma_1\sigma_2\neq 0$. Since the rows of $c_2$ are orthogonal, its first and second rows give
$$
F_{14,21}F_{14,22}(\frac{1}{\sigma_1^2}-\frac{1}{\sigma_2^2})=0,
$$
so that $\sigma_1=\sigma_2$, or $\mu^2=\nu^2$, if $F_{14,21}F_{14,22}\neq 0$; by a coordinate sign change we may assume
$\mu=-\nu.$
Moreover, with $\mu=-\nu$, the second and the fourth rows of $c_2$, being orthogonal, yields
$$
F_{14,21}(F_{11,22}-F_{14,23})=0,
$$
so that 
$F_{11,22}=F_{14,23}$ 
when $F_{14,21}F_{14,22}\neq 0$; the second and third rows then give 
$$
F_{14,22}\;\mu(\alpha-\gamma)=0,
$$
from which we conclude 
$$
\alpha=\gamma,\;\;\text{under the condition}\;\; \mu\neq 0\;\text{and}\;F_{14,21}F_{14,22}\neq 0.
$$

Continue to assume $F_{14,21}F_{14,22}\neq 0$. If $\mu=0$ then $\mu=\nu=0$ and so $\sigma_1=\sigma_2=1$. The lengths of the the first two rows of $F_2$ being 1 implies
\begin{equation}\label{norm}
F_{11,22}^2+F_{14,21}^2+F_{14,22}^2+\alpha^2=1,\quad F_{11,22}^2+F_{14,21}^2+F_{14,22}^2+\gamma^2=1,
\end{equation}
so that we obtain $\alpha^2=\beta^2$. After a coordinate sign change, we may assume
$\alpha=\gamma$
since $\mu=-\nu=0$ is not affected.

If $F_{14,22}=0$ and $F_{14,21}\neq 0$, we first establish 
$\mu^2=\nu^2.$

The orthogonality of the first and the third rows, and respectively the second and the fourth rows of $c_2$ give, after cancelling $F_{14,21}$,
\begin{equation}\label{112}
\aligned
&\frac{1}{\sigma_2^2}F_{11,22}-\frac{1}{\sigma_1^2}F_{14,23}=\alpha(\frac{\mu}{\sigma_1^2}+\frac{\nu}{\sigma_2^2}),\\
&\frac{1}{\sigma_1^2}F_{11,22}-\frac{1}{\sigma_2^2}F_{14,23}=-\gamma(\frac{\mu}{\sigma_1^2}+\frac{\nu}{\sigma_2^2}),
\endaligned
\end{equation}
from which we solve to obtain
\begin{equation}\label{two}
\aligned
&(\frac{\sigma_1^2}{\sigma_2^2}-\frac{\sigma_2^2}{\sigma_1^2})F_{11,22}=(\alpha\sigma_1^2+\gamma\sigma_2^2)(\frac{\mu}{\sigma_1^2}+\frac{\nu}{\sigma_2^2}),\\
&(\frac{\sigma_1^2}{\sigma_2^2}-\frac{\sigma_2^2}{\sigma_1^2})F_{14,23}=(\alpha\sigma_2^2+\gamma\sigma_1^2)(\frac{\mu}{\sigma_1^2}+\frac{\nu}{\sigma_2^2}).
\endaligned
\end{equation}
Since the rows of $F_2$ are of unit length, we have
\begin{equation}\label{four}
\aligned
&\frac{\sigma_1^2}{\sigma_2^2}(F_{11,22}-\alpha\nu)^2-\sigma_1^2(1-\alpha^2)=-F_{14,21}^2,\\
&\frac{\sigma_2^2}{\sigma_1^2}(F_{11,22}+\gamma\mu)^2-\sigma_2^2(1-\gamma^2)=-F_{14,21}^2,\\
&\frac{\sigma_1^2}{\sigma_2^2}(F_{14,23}-\gamma\nu)^2-\sigma_1^2(1-\gamma^2)=-F_{14,21}^2,\\
&\frac{\sigma_2^2}{\sigma_1^2}(F_{14,23}+\alpha\mu)^2-\sigma_2^2(1-\alpha^2)=-F_{14,21}^2.
\endaligned
\end{equation}
Substituting~\eqref{two} into~\eqref{four} and cancelling $F_{14,21}^2$ by subtraction, we derive, respectively, for the first and second pairs in~\eqref{four}
\begin{equation}\label{gd}
\aligned
&(\,(1-\alpha^2)\sigma_1^2-(1-\gamma^2)\sigma_2^2\,)\,(\,\frac{\sigma_1^2}{\sigma_2^2}-\frac{\sigma_2^2}{\sigma_1^2}\,)^2=(\frac{\sigma_1^2}{\sigma_2^2}-\frac{\sigma_2^2}{\sigma_1^2})\,(\alpha^2-\gamma^2)\,(\mu^2-\nu^2),\\
&(\,(1-\alpha^2)\sigma_2^2-(1-\gamma^2)\sigma_1^2\,)\,(\,\frac{\sigma_1^2}{\sigma_2^2}-\frac{\sigma_2^2}{\sigma_1^2}\,)^2=(\frac{\sigma_1^2}{\sigma_2^2}-\frac{\sigma_2^2}{\sigma_1^2})\,(\alpha^2-\gamma^2)\,(\mu^2-\nu^2).
\endaligned
\end{equation}
Suppose $\sigma_1\neq \sigma_2$. We equate the left hand sides and cancel the common fraction to see
$$
(2-\alpha^2-\gamma^2)\sigma_1^2=(2-\alpha^2-\gamma^2)\sigma_2^2,
$$
which gives
$$
2=\alpha^2+\gamma^2,
$$
so that in fact
$$
\alpha^2=\gamma^2=1,
$$
which implies that the first row of $c_2$ is identically zero, so that $F_{14,21}=0$, a contradiction. Thus $\sigma_1=\sigma_2$, i.e., $\mu^2=\nu^2$, so that after a coordinate sign change we may assume
\begin{equation}\label{equal}
\mu=-\nu.
\end{equation}
It follows from~\eqref{112} that
\begin{equation}\label{simeq}
F_{11,22}=F_{14,23}.
\end{equation}
Meanwhile, since each row of $F_2$ is of unit length, we calculate the first and the second to see
\begin{equation}\label{na}
2\alpha\mu F_{11,22}+(F_{11,22}^2+F_{14,21}^2)=\sigma^2-\alpha^2,\quad 2\gamma\mu F_{11,22}+(F_{11,22}^2+F_{14,21}^2)=\sigma^2-\gamma^2,
\end{equation}
with $\sigma^2=1-\mu^2$. Cancelling out $F_{11,22}^2+F_{14,21}^2$ we obtain
$$
-2(\alpha-\gamma)\mu F_{11,22}=\alpha^2-\gamma^2.
$$
In particular, if $\alpha\neq \gamma$, then
\begin{equation}\label{useful}
-2\mu F_{11,22}=\alpha+\gamma.
\end{equation}
Substituting it into the sum of the two identities in~\eqref{na} we obtain
\begin{equation}\label{yum}
F_{14,21}^2+F_{11,22}^2=\sigma^2+\alpha\gamma.
\end{equation}

There are indeed orthogonal multiplications in this category for which $\alpha\neq \gamma$. For instance, let us assume $\alpha=0$ and $\gamma\neq 0$. Then~\eqref{useful} and~\eqref{yum} assert that $F_{14,21}$ exists so long as 
$$
\sigma^2\geq  \frac{\gamma^2}{4\mu^2}, \quad \text{or}\quad 4\mu^2(1-\mu^2)\geq \gamma^2,
$$
for which there are $\gamma$ once $\mu$ is chosen appropriately.

Now, observe that, similar to~\eqref{dual},  when we apply the orthogonal element in $SO(4)$ that maps $f_1$ and $-f_4$, $f_4$ to $f_1$ and fixes both $f_2$ and $f_3$, we interchange the first two columns of both the self-dual and anti-self-dual parts of the Parker matrix such that
$$
\aligned
&f_1\wedge f_2+f_3\wedge f_4\rightarrow -(f_1\wedge f_3+f_4\wedge f_2)\\
&f_1\wedge f_3+ f_4\wedge f_2\rightarrow f_1\wedge f_2+f_3\wedge f_4,\\
&f_1\wedge f_4+f_2\wedge f_3\rightarrow f_1\wedge f_4+f_2\wedge f_3,
\endaligned
$$
and likewise for anti-self-dual forms. Moreover, we interchange $e_2$ and $e_3$ while fixing $e_1$, so that the transformation

\begin{equation}\label{bin}
\aligned
&e_1 \rightarrow e_1,\quad e_2\rightarrow e_3,\quad e_3\rightarrow e_2,\\
&f_1\rightarrow -f_4,\quad f_4\rightarrow f_1,\quad f_2\rightarrow f_2,\quad f_3\rightarrow f_3,
\endaligned
\end{equation}
retains the diagonal-and-upper-triangular pattern. With this, the case when $\beta=0=F_{14,22}$ and $F_{14,21}\neq 0$ is converted to the case where $\beta\neq 0$ and $F_{14,21}=F_{14,22}=0$ given in Proposition~\ref{prop}. 
%%Similar to~\eqref{trans}, we have the conversion
%%\begin{equation}\label{uum}
%%\alpha=F_{31,13}\longleftrightarrow F_{11,22},\quad \gamma=F_{14,32}\longleftrightarrow F_{14,23}.
%%\end{equation}

This is the symmetry of the moduli we will explore next. 
We denote the resulting $F$-quantities obtained through the transformation~\eqref{bin} with an extra * to avoid confusion. We have
\begin{equation}\label{1122}
\aligned
&\mu^*=F^*_{24,31}=-F_{31,24}=-\mu,\quad \nu^*=F^*_{23,32}=F_{33,22}=-\nu,\\
&\beta^*=F^*_{31,14}=-F_{24,11}=F_{14,21},\quad \alpha^*=F_{31,13}^*=-F_{24,13}=F_{14,23},\\
&\gamma^*=F_{22,11}=F_{11,22},\quad F^*_{11,22}=-F_{14,32}=-\gamma,\quad F^*_{14,23}=F_{11,33}=-\alpha.
\endaligned
\end{equation}

If we invoke Proposition~\ref{prop}, where $\mu^*=-\nu^*$ and $\alpha^*=\gamma^*$, we again conclude~\eqref{equal} and~\eqref{simeq} obtained through algebraic means. Meanwhile, the class 
$$
F^*_{14,23}=-F^*_{11,22}-2\alpha^*\mu^*
$$
in~\eqref{qq} now translates into
$$
-\alpha=\gamma+2\mu F_{14,23},
$$
which is exactly~\eqref{useful} in view of~\eqref{yum}.
\end{proof}

\begin{corollary}\label{corollary} The case when $\beta=0=F_{14,22}$ and $F_{14,21}\neq 0$ {\rm (}or $\beta=0=F_{14,21}$ and $F_{14,22}\neq 0${\rm )} is equivalent to the generic case when $\beta\neq 0$ and $F_{14,21}=F_{14,22}=0$ in the grand moduli ${\mathcal G}_{(3,4,8)}$.
Therefore, the latter are boundary points of the coarse fundamental domain identified with the former to belong to the grand moduli.

 \end{corollary}
 
 \begin{proof} This is through the symmetry~\eqref{bin} in the preceding proof.
  \end{proof}

 \begin{remark}\label{rmk1} Note that in the grand moduli, $c_1=\sigma\, Id$, $c_2$ is of the form
$$
c_2=-\frac{\beta\mu}{\sigma}\, Id+ M,
$$
where $M$ is skew-symmetric since $\alpha=\gamma$ and $\mu=-\nu$. The Hurwitz condition 
$$
c_2c_2^{tr}+w_2w_2^{tr}=Id
$$
is reduced to
$$
MM^{tr}=\theta^2\,Id,\quad \theta=\sqrt{1-\frac{\beta^2\mu^2}{\sigma^2}-\alpha^2},\quad\sigma=\sqrt{1-\mu^2}.
$$
Meanwhile, $w_1$ and $w_2$ are of the forms
$$
w_1=\mu J, \quad w_2=\tau L,\quad \tau=\sqrt{\alpha^2+\beta^2},
$$ where $J, L$ are skew-symmetric and orthogonal. Therefore, matrices in the grand moduli are, up to adjoint equivalence, of the form
$$
\aligned
&F_1=\begin{pmatrix}\sigma\,Id&\mu J\end{pmatrix},\quad F_2=\begin{pmatrix}-\frac{\beta\mu}{\sigma}\,Id+\theta M&\tau\,L\end{pmatrix},\quad F_3=\begin{pmatrix}0&Id\end{pmatrix},\\ 
\endaligned
$$
satisfying
$$
\mu\tau(JL+LJ)=2\beta\mu\,Id.
$$
\end{remark}

\subsection{The degenerate case $\beta=F_{14,21}=F_{14,22}=0$}\label{most} This is the case when the upper triangular $T$ in~\eqref{0} is also diagonal. It includes the case when either $\sigma_1$ or $\sigma_2$ is zero as given in~\eqref{degenerate}.

It is more convenient to denote
$$
\alpha=\cos(\phi),\quad \gamma=\cos(\psi), \quad \mu=\cos(\theta),\quad \nu=\cos(\eta), \quad 0\leq \phi,\psi,\theta,\eta\leq \pi.
$$
Then we have $\sigma_1=\sin(\theta),\sigma_2=\sin(\eta)$, and $c_2$ in~\eqref{grand} is
$$
c_2=\begin{pmatrix} 0&\pm\sin(\phi)&0&0\\\pm\sin(\psi)&0&0&0\\0&0&0&\pm\sin(\phi)\\0&0&\pm\sin(\psi)&0\end{pmatrix},
$$
with
\begin{equation}\label{frac}
\aligned
&\frac{F_{11,22}-\alpha\nu}{\sin(\eta)}=\pm \sin(\phi),\quad \frac{-F_{11,22}-\gamma\mu}{\sin(\theta)}=\pm\sin(\psi),\\
&\frac{-F_{14,23}-\alpha\mu}{\sin(\theta)}=\pm\sin(\phi),\quad \frac{F_{14,23}-\gamma\nu}{\sin(\eta)}=\pm \sin(\psi).
\endaligned
\end{equation}
We obtain
\begin{equation}\label{be}
F_{11,22}=\cos(\phi\pm\eta)=-\cos(\psi\pm\theta),\quad F_{14,23}=\cos(\psi\pm\eta)=-\cos(\phi\pm\theta).
\end{equation}

This category is where a large set of anomaly occurs for which $\alpha^2\neq\gamma^2$ and $\mu^2\neq \nu^2$. 

The orthogonal multiplications in this category that satisfy $\alpha=\gamma$ and $\mu=-\nu$ to belong to the grand moduli is when 
\begin{equation}\label{dge}
\cos(\phi)=\cos(\psi),\quad \cos(\theta)=-\cos(\eta),
\end{equation}
so that
$$
\phi=\psi,\quad \theta+\eta=\pi
$$
with the constraint
$$
\cos(\phi\pm\eta)=-\cos(\phi\pm\theta).
$$
More generally, the condition $\alpha^2=\gamma^2$ and $\mu^2=\nu^2$ is equivalent to 
\begin{equation}\label{degree}
\sin(\eta)=\sin(\theta),\quad \sin(\phi)=\sin(\psi)
\end{equation}
with the constraint~\eqref{be}.
%%\begin{equation}\label{cos}
%%\aligned
%%&\sin(\eta)\sin(\phi)=\sin(\theta)\sin(\psi),\quad \sin(\theta)\sin(\phi)=\sin(\eta)\sin(\psi),\\
%%&\cos(\phi-\eta)=-\cos(\psi+\theta),\quad \cos(\psi-\eta)=-\cos(\phi+\theta).
%%\endaligned
%%\end{equation}
%%It follows that in the case when none of the angles are $0$ or $\pi$, the first pair gives us
%%\begin{equation}\label{uni}
%%\sin(\phi)=\sin(\psi),\quad \sin(\eta)=\sin(\theta).
%%\end{equation}
%%Note that~\eqref{degenerate} is exactly the case when either $\sin(\theta)$ or $\sin(\eta)=0$. Assume $\sin(\theta),\sin(\eta)\neq 0$, then we have $\sin(\phi)=0$ if and only if $%%\sin(\psi)=0$, in which case~\eqref{cos} leads to $\cos(\eta)=\pm\cos(\theta).$ That is,
%%\begin{equation}\label{degree}
%%\sin(\eta)=\sin(\theta),\quad \sin(\phi)=\sin(\psi).
%%\end{equation}
%%So, uniformly~\eqref{uni} is the condition for $c_2$ to be skew-symmetric. One none trivial family is when $\eta=\theta$ and $\phi+\psi=\pi$, or $\phi=\psi$ and $\theta\eta=\pi$.

A family that does not satisfy~\eqref{degree} is when $\theta+\eta+\phi+\psi=\pi$.

Putting Corollary~\ref{corollary} and the discussion in this section, we obtain the following.

\begin{corollary}\label{roc} The moduli space of orthogonal multiplications of type $[3,4,8]$ consists of two components. The $5$-dimensional grand moduli ${\mathcal G}_{(3,4,8)}$ given in Proposition~{\rm \ref{prop}} with the constraints~\eqref{mod1},~\eqref{mod2},~\eqref{14}, and the $3$-dimensional degenerate moduli ${\mathcal X}_{(3,4,8)}$ given in~\eqref{frac} with the constraint~\eqref{be}. They intersect at the points where~\eqref{dge} holds.
\end{corollary}

\begin{remark} Corollary~{\rm \ref{roc}} is reminiscent of the Cartan Umbrella $z(x^2+y^2)=x^3$. It is a real irreducible variety for which the umbrella canopy is $2$-dimensional, similar to the grand moduli ${\mathcal G}_{(3,4,8)}$, and the umbrella shaft is the $1$-dimensional $z$ axis, similar to the anomalous ${\mathcal X}_{(3,4,8)}$. A property on the canopy need not hold on the shaft.
\end{remark}

\subsection{Type $[3,4,7]$}\label{743} As a consequence of Section~\ref{most}, when the orthogonal multiplication is of type $[3,4,7]$, we know $\sigma_1\sigma_2=0$, because otherwise, 
$c_1$ and $w_3$ would account for 8 dimensions, not 7. It follows by~\eqref{degenerate} that the orthogonal multiplication belongs to the degenerate case in Section~\ref{most}.

We may assume $\sigma_1=0$ and $\sigma_2\neq 0$ up to range equivalence, i.e., $\mu^2=1$ and $\nu^2<1$, then already $c_1$ and $w_3$ account for 6 dimensions, so that either the first or the fourth column of $c_2$ is zero for the range dimension to be 7. We may assume it is the first column that vanishes up to range equivalence, i.e., $\sin(\psi)=0$, or, $\gamma^2=1$. The Hurwitz condition dictates
$$
\pm \sin(\phi)=-\alpha\nu-\gamma\mu,\quad \alpha\mu+\gamma\nu=0,
$$
 from which we conclude
 $$
 \sin(\phi)=\sigma_2,\quad \mu^2=\gamma^2=1, \quad \alpha^2=\nu^2,\quad \sigma_2=\sqrt{1-\nu^2}.
 $$
 The moduli dimension of such orthogonal multiplications is 1.

\section{A more rigid range equivalence and its application to isoparametric hypersurfaces}\label{reverse} Recall the process leading to~\eqref{c1} and~\eqref{grand}. Once we normalize the Parker matrix as given in~\eqref{0}, we look at the span $V$ of $e_3\circ f_1,\cdots,e_3\circ f_4$ and its orthogonal complement $V^\perp$, relative to which we set the anchor matrix $F_3$ to be $\begin{pmatrix}0&Id\end{pmatrix}$. From $F_3$ we can set $F_1$ in the canonical form~~\eqref{c1} in agreement with orthogonal multiplications of type $[2,4,8]$ in Section~\ref{s2}, and consequently build $F_2$ in terms of the data ~\eqref{1} arising from the Parker matrix normalization. 

Consider the case when the range equivalence is more rigidly restricted to a fixed decomposition ${\mathbb R}^8={\mathbb R}^4\oplus{\mathbb R}^4$, so that only $SO(4)\oplus SO(4)\subset SO(8)$ is allowed. Accordingly, an orthogonal multiplication $F$ in the coarse fundamental domain, with $F_3=\begin{pmatrix}0&Id\end{pmatrix}$
obtained by the $SO(8)$ range equivalence, may turn into one for which the first block of $F_3$ is nonzero when imposing the more rigid $SO(4)\oplus SO(4)$ range equivalence. That is, $V$ need not be the prescribed second copy of ${\mathbb R}^4$, even when $F_1$ relative to the fixed decomposition is the prescribed form as in~\eqref{c1}. 

As an example, consider the generic case in Proposition~\ref{prop}. We know $|\mu|=|\nu|<1$ by Lemma~\ref{l3} so that $\sigma=\sigma_1=\sigma_2>0$. Set $\lambda:=|\mu|/\sigma$ and $J:=w_1/\sigma$ with $w_1$ given in~\eqref{c1}. Multiplying the orthogonal matrix
$$
\aligned
&O:=\begin{pmatrix}A&B\\C&D\end{pmatrix},\quad A=I-JC,\quad B=J-JD, \quad C=DJ-J,\\
&D=-D^{tr},\quad DD^{tr}=(1-\lambda^2)/(1+\lambda^2)\; Id
\endaligned
$$
on the right of $F_1, F_2, F_3$ in~\eqref{c1} and~\eqref{grand}, we see 
$$
F_1^*:=F_1\, O=F,\quad F_2^*:=F_2\, O, \quad F_3^*:= \begin{pmatrix} C&D\end{pmatrix}.
$$
That is, $F_1^*$ of the orthogonal multiplication $F^*$, relative to the fixed decomposition of ${\mathbb R}^8$, assumes the same prescribed form as given in~\eqref{c1}. However,
$V$, the span of the rows of $F_3^*$, is not in general the second copy of ${\mathbb R}^4$ in the fixed decomposition of ${\mathbb R}^8$; performing basis change over each ${\mathbb R}^4$ summand, which amounts to multiplying on the right of $O$ by an orthogonal matrix in the diagonal block form, does not convert $O$ to the identity matrix in general. 

%%ask whether the process can be reversed. Namely, given a fixed decomposition of ${\mathbb R}^8={\mathbb R}^4\oplus{\mathbb R}^4$ and the Parker normalization,
%% relative to which we have 
%%$F_1=\begin{pmatrix}c_1&w_1\end{pmatrix}$ in~\rm \eqref{c1}. 
%%what sort of orthogonal multiplications can we capture in such a way that the image of $F_3$ lands in the second copy of ${\mathbb R}^4$ in the prescribed decomposition of %%${\mathbb R}^8$, when $F_1$ given in~\eqref{c1} also respects the decomposition?

%%Note that the degenerate case in Section~{\rm \ref{most}} was in fact more conveniently constructed along this reverse fashion, where in the construction we simply set %%$F_3=\begin{pmatrix}0&Id\end{pmatrix}$ once $F_1$ and $F_2$ are determined. 
%%This gives us a coarse fudamental domain of ${\mathcal X}_{(3,4,8)}$ when the anchor matrix is chosen to be $F_3$,
%%but it does not mean that the rows of $F_3$ live in the second copy of ${\mathbb R}^4$ of the prescribed decomposition. Rather, it only says that the span $V$ of the rows of %%$F_3$, or of $e_3\circ f_1, \cdots, e_3\circ f_4$, are isomorphic to the second copy of ${\mathbb R}^4$ via the isomorphisms 
%%$$
%%{\mathbb R}^4\oplus{\mathbb R}^4\simeq{\mathbb R}^8\simeq V^{\perp}\oplus V
%%$$
%%through range equivalence. 

In a similar fashion, given a degenerate orthogonal multiplication $F$ representing ${\mathcal X}_{(3,4,8)}$ in the coarse fundamental domain, assume that $F_1$ is as prescribed in~\eqref{c1} relative to the fixed decomposition of ${\mathbb R}^8$.
To determine that $V$ is the fixed second copy of ${\mathbb R}^4$, we now choose $F_1$ as the anchor matrix in place of $F_3$ to utilize the existing symmetry. The data in~\eqref{GG} below then dictates that $V$ is the fixed second copy of ${\mathbb R}^4$. Meanwhile, $F_2$ is now represented in two coarse fundamental domains of ${\mathcal X}_{(3,4,8)}$, one for which $F_3$ is the anchor matrix and the other for which $F_1$ is. Thus, it is only when the orthogonal multiplication belongs to the intersection of these two coarse fundamental domains can we assert that $V$ is the fixed second copy of ${\mathbb R}^4$. We carry out the details as follows.

\begin{corollary}\label{GD} Notations and conditions as above, assume $\sigma_1\sigma_2\neq 0$ for the degenerate case presented in Section~{\rm \ref{most}}. Then $V$ is the fixed second copy of ${\mathbb R}^4$ only when the orthogonal multiplication belongs to the grand moduli ${\mathcal G}_{(3,4,8)}$.
\end{corollary}

\begin{proof} We utilize the symmetry~\eqref{bin} explicitly, letting $F_1$ be the ``anchor matrix'' instead so that $F_1=\begin{pmatrix}0&Id\end{pmatrix}$. We calculate to see, relative to $F_1$,
\begin{equation}\label{GG}
F_3=\begin{pmatrix}c_3&w_3\end{pmatrix},\quad c_3=\begin{pmatrix}\sigma_1&0&0&0\\0&\sigma_2&0&0\\0&0&\sigma_2&0\\0&0&0&\sigma_1\end{pmatrix},\quad w_3=\begin{pmatrix} 0&0&0&\mu\\0&0&\nu&0\\0&-\nu&0&0\\-\mu&0&0&0\end{pmatrix},
\end{equation}
and $F_2=\begin{pmatrix}c_2&w_2\end{pmatrix},$ where
$$
\aligned
&c_2=\begin{pmatrix}F_{14,21}\mu/\sigma_1&0&(F_{11,22}-\alpha\nu)\nu/\sigma_2&0\\0&-F_{14,21}\nu/\sigma_2&0&-(F_{11,22}+\gamma\mu)\mu/\sigma_1\\(F_{14,23}+\alpha\mu)\mu/\sigma_1&0&-F_{14,21}\nu/\sigma_2&0\\0&-(F_{14,23}-\gamma\nu)\nu/\sigma_2&0&F_{14,21}\mu/\sigma_1\end{pmatrix},\\
&w_2=\begin{pmatrix}0&F_{11,22}&0&-F_{14,21}\\-F_{11,22}&0&F_{14,21}&0\\0&-F_{14,21}&0&-F_{14,23}\\F_{14,21}&0&F_{14,23}&0\end{pmatrix}.
\endaligned
$$
~\eqref{bin} tells us to interchange the first and fourth columns and rows, we end up with
\begin{equation}\label{imp}
c_3=\begin{pmatrix}\sigma_1&0&0&0\\0&\sigma_2&0&0\\0&0&\sigma_2&0\\0&0&0&\sigma_1\end{pmatrix},\quad w_3=\begin{pmatrix} 0&0&0&-\mu\\0&0&\nu&0\\0&-\nu&0&0\\\mu&0&0&0\end{pmatrix},
\end{equation}
\begin{equation}\label{IMP}
\aligned
&c_2=\begin{pmatrix}F_{14,21}\mu/\sigma_1&-(F_{14,23}-\gamma\nu)\nu/\sigma_2&0&0\\-(F_{11,22}+\gamma\mu)\mu/\sigma_1&-F_{14,21}\nu/\sigma_2&0&0\\0&0&-F_{14,21}\nu/\sigma_2&(F_{14,23}+\alpha\mu)\mu/\sigma_1\\0&0&(F_{11,22}-\alpha\nu)\nu/\sigma_2&F_{14,21}\mu/\sigma_1\end{pmatrix},\\
&w_2=\begin{pmatrix}0&0&F_{14,23}&F_{14,21}\\0&0&F_{14,21}&-F_{11,22}\\-F_{14,23}&-F_{14,21}&0&0\\-F_{14,21}&F_{11,22}&0&0\end{pmatrix},\quad \mu=-\nu, \;\sigma_1=\sigma_2,\;\; F_{11,22}=F_{14,23}.
\endaligned
\end{equation}
Comparing~\eqref{imp} and~\eqref{IMP} with~\eqref{c1} and~\eqref{grand}, upon which we need to change the sign of both the first column and row in~\eqref{imp} and~\eqref{IMP}, dictated by~\eqref{bin}, to make the $\pm$ signs of the two sets of $c$- and $w$-matrices agreeable, we conclude, where we denote the new $F$-quantities with an extra *, that we obtain the same identities,
$$
\mu^*=-\mu,\quad \nu^*=-\nu,\quad \beta^*=F_{14,21},\quad \alpha^*=F_{14,23},\quad\gamma^*=F_{11,22}
$$
as in~\eqref{1122}. Furthermore, we have
\begin{equation}\label{*}
\aligned
&F_{11,22}^*-\alpha^*\nu^*=-(F_{14,23}-\gamma\nu)\nu,\quad F^*_{11,22}+\gamma^*\nu^*=(F_{11,22}+\gamma\mu)\mu\\
&-F_{14,23}^*-\alpha^*\mu^*=(F_{14,23}+\alpha\mu)\mu,\quad F_{14,23}^*-\gamma^*\nu^*=(F_{11,22}-\alpha\nu)\nu.
\endaligned
\end{equation}
Incorporating~\eqref{1122} and~\eqref{*},  we see
\begin{equation}\label{how}
\gamma=\gamma\nu^2,\quad \gamma=\gamma\mu^2,\quad \alpha=\alpha\nu^2,\quad \alpha=\alpha\mu^2.
\end{equation}
If either $\alpha$ or $\gamma$ is nonzero, then $\mu^2=\nu^2=1$ so that $\sigma_1=\sigma_2=0$, contradicting $\sigma_1\sigma_2\neq 0$. 
%%it follows from~\eqref{degenerate} that
%%$$
%%\alpha\nu=-\gamma\mu,\quad -\alpha\mu=\gamma\nu,
%%$$
%%from which there follows $\alpha=\gamma$.

Otherwise, $\alpha=\gamma=0$. The rows of $c_2$ in~\eqref{grand}, being of unit length, then gives $\sigma_1=\sigma_2$, or, $\mu^2=\nu^2$, so that we may assume $\mu=-\nu$ without affecting $\alpha=\beta=0$. In particular, $c_2$ in~\eqref{frac} is skew-symmetric and nonzero. 
\end{proof}

\begin{remark} Although the category where $\beta=0=F_{14,22}$ and $F_{14,21}\neq 0$ can be interchangeably converted to the category where $\beta\neq 0$ and $F_{14,21}=F_{14,22}=0$
that lives in the grand moduli, we can also see directly that the conclusion of Corollary~{\rm \ref{GD}} holds as well for both categories as follows. 

Assume $\beta=0=F_{14,22}$ and $F_{14,21}\neq 0$.
We know $\mu=-\nu$ by~\eqref{equal}, so that we slightly modify~\eqref{how} to conclude
$$
\gamma=\gamma\mu^2,\quad \alpha=\alpha\mu^2.
$$
Now $\mu^2\neq 1$, since otherwise $\sigma=0$ so that $F_{14,21}=0$ by~\eqref{degenerate}, which is absurd. Thus, $\alpha=\gamma=0$. But then the rows of $c_2$ in~\eqref{grand}, being of unt length, implies that $\mu^2=\nu^2$, so that we may assume $\mu=-\nu$ without affecting $\alpha=\gamma=0$. In particular, $c_2$ in~\eqref{grand} is skew-symmetric and nonzero.

On the other hand, assume $\beta\neq 0$ and $F_{14,21}=F_{14,22}=0$. Going through the same arguments we obtain $\alpha=\alpha\mu^2$ since we know $\alpha=\gamma$ and $\mu=-\nu$ by Proposition~{\rm \ref{prop}}, so that $\alpha=0=\gamma$ and so $F_{11,22}^2=F_{14,23}^2$ by~\eqref{qq}. Note that in this case, $c_2$ in~\eqref{IMP} is slightly modified with $-\beta\mu^2/\sigma,\beta\nu^2/\sigma,-\beta\nu^2/\sigma,\beta\mu^2/\sigma$ filling the $(1,4)$-, $(2,3)$-, $(3,2)$-, and $(4,1)$-entries, respectively, where $\sigma=\sigma_1=\sigma_2$. Suppose $F_{11,22}=F_{14,23}=0$, then the modified~\eqref{IMP} implies
$\beta\mu^2=\pm \sigma$ since each row of $F_2$ is of unit length and, except for one entry, all other entries of its first row are zero; however,~\eqref{grand} implies $\beta^2=\sigma^2$ for the same reason, so that $\mu^2=1$, which is absurd as then $0=\sigma=\beta\neq 0$. Hence, $F_{11,22}^2=F_{14,23}^2$ are nonzero. In particular, 
$c_2$ in~\eqref{grand} is of the form $a\,Id + C$ for some real number $a$ and some nonzero skew-symmetric $C$.
\end{remark}

To make a long story short, let us remark that orthogonal multiplications of type $[3,4,8]$ play a decisive role in the classification of isoparametric hypersurfaces with four principal curvatures and multiplicity pair $(7,8)$ in $S^{31}$~\cite[Section 7]{C}. 
The $SO(4)\oplus SO(4)$ range equivalence is important in~\cite[Section 7]{C} because the orthogonal multiplication must respect a prescribed ${\mathbb R}^4\oplus{\mathbb R}^4$ decomposition intrinsic to the underlying isoparametric structure that is associated with the focal manifolds of the hypersurface, in such a way that, relative to this intrinsic decomposition, $F_1$ is prescribed as in~\eqref{c1} and the span of the row of $F_3$ is the second copy of ${\mathbb R}^4$.
What is developed in the preceding section and this section asserts, in the isoparametric situation, where necessarily $\sigma_1\sigma_2\neq 0$, that $\alpha=\gamma$ and $\mu=-\nu$ uniformly, and, moreover, $c_2$ in~\eqref{grand} is of the form
$$
c_2=a\,Id+C,\quad a\in{\mathbb R},\quad 0\neq C\;\,\text{is skew-symmetric},
$$ 
so that, up to adjoint equivalence,
$$
c_2=a\,Id+b\begin{pmatrix}I&0\\0&\pm I\end{pmatrix},\quad I=\begin{pmatrix}0&-1\\1&0\end{pmatrix},\;\; b\neq 0,
$$ 
which is pivotal for establishing the decisive Corollary 7.3 in~\cite{C}.

\section{Moduli space of type $[3,4,p], p\leq 12$}\label{finali}  A full orthogonal multiplication of type $[3,4,12]$ has a 9-dimensional moduli in $\wedge^2{\mathbb R}^3\otimes\wedge^2{\mathbb R}^4/SO(3)\otimes SO(4)$. We can see this explicitly. Namely, the process to carry out~\eqref{c1} and~\eqref{grand} goes through verbatim. Next, we complete the orthonormal basis of ${\mathbb R}^{12}$ from 
$u_1,\cdots,u_8$ constructed in Section~\ref{4.1} by augmenting four basis vectors $u_{-4},u_{-3},u_{-2},u_{-1}$ such that

\begin{equation}\label{new}
\aligned
&\tilde{F}_3=\begin{pmatrix}0&0&Id\end{pmatrix},\quad \tilde{F}_1=\begin{pmatrix} 0&c_1&w_1\end{pmatrix},\quad \tilde{F}_2=\begin{pmatrix}\epsilon_2&c_2&w_2\end{pmatrix},\\
%%&e_2:=\begin{pmatrix}\tau_{11}&\tau_{12}&\tau_{13}&\tau_{14}\\ 0&\tau_{22}&\tau_{23}&\tau_{24}\\ 0&0&\tau_{33}&\tau_{34}\\ 0&0&0&\tau_{44}\end{pmatrix}
\endaligned
\end{equation}
with $c_1, c_2, w_1, w_2$ given in~\eqref{grand}. 

It is more convenient to use the normal exponential map to parametrize $\tilde{F}_2$. Namely, we let the ${\mathbb R}^8$ containing the rows of $\begin{pmatrix}c_2&w_2\end{pmatrix}$ be horizontal so that the ${\mathbb R}^4$ containing the rows of $\epsilon_2$ is vertical. For any vector $v$, identified with a vector in the vertical ${\mathbb R}^4$, normal to the horizontal $S^7$ at each of its point $x$, we have the map
\begin{equation}\label{exp}
Exp:(x,v)\mapsto \cos(|v|) \;x +\sin(|v|)\; v/|v|, \quad (p,0)\mapsto p,
\end{equation}
which is a diffeomorphism from $S^7\times D$ into $S^{11}$ describing a tubular neighborhood of $S^7$ in $S^{11}$, where $D$ is a sufficiently small disk around 0 in ${\mathbb R}^4$.  In light of this, the $j$th row of $\tilde{F}_2$ can be written as follows. In~\eqref{new}, let $z_j$ be the $j$th row of $\begin{pmatrix}c_2&w_2\end{pmatrix}$ and $y_j$ be the $j$th row of $\epsilon_2$. Let 
$$
x_j:=z_j/|z_j|,\quad \sin(\omega_j):=|y_j|,\quad v_j: =\omega_j y_j/|y_j|.
$$ 
Then 
$$
Exp(x_j,v_j)=(y_j, z_j)=j\text{th row of}\;\tilde{F}_2.
$$
We have the map
\begin{equation}\label{pi}
\Pi:\tilde{F}_2\mapsto (\, |\pi\circ Exp^{-1}((y_1,z_1))|,\,\cdots,\, |\pi\circ Exp^{-1}((y_1,z_1))|\,)=(\omega_1,\omega_2,\omega_3,\omega_4),
\end{equation}
where $\pi:(x,v)\mapsto v.$ Note that the angles $\omega_j$, though set to be nonnegative in~\eqref{exp}, can in fact be extended to negative values since 
$$
Exp(p,-v)=\cos(-|v|) \;p +\sin(-|v|)\; v/|v|,
$$
which we adopt henceforth. 

The moduli of type $[3,4,12]$ is parametrized by the nine generic variables 
\begin{equation}\label{nine}
\alpha, \beta, \gamma, \mu, \nu, F_{14,21}, F_{14,22}, F_{11,22}, F_{14,23}.
\end{equation}
Let ${\mathcal V}$ be the span of the rows of $\tilde{F}_1,\tilde{F}_2,\tilde{F}_3$. We have 
$$
p=\dim({\mathcal V}).
$$

Suppose $\sigma_1\sigma_2\neq 0$. Then it is expected that generically $p=12-n$, where $n$ is the number of zeros of $v_1,\cdots,v_4.$ We studied extensively the important case when $n=4$ in the previous sections, while $n=0$ gives the dimension of the entire moduli. The expectation is indeed the case. We next show that there is a moduli of type $[3,4,12-n]$ for each $1\leq n\leq 3$ under the generic assumption that $\sigma_1\sigma_2\neq 0$. To this end, observe that in the degenerate case $\beta=F_{14,21}=F_{14,22}=0$, we have, as in Section~\ref{most}, $\tilde{F}_2=\begin{pmatrix}\epsilon_2&c_2&w_2\end{pmatrix}$, where

\begin{equation}\label{tai}
\aligned
&\begin{pmatrix}\epsilon_2&c_2\end{pmatrix} =\\
&\begin{pmatrix}s\sin(\phi)&0&0&0&0&s\cos(\phi)&0&0\\0& t\sin(\psi)&0&0& t\cos(\psi)&0&0&0\\0&0&s\sin(\zeta)&0&0&0&0&s\cos(\zeta)\\0&0&0&t\sin(\xi)&0&0&t\cos(\xi)&0\end{pmatrix}.
\endaligned
\end{equation}
for some angles $\phi, \psi, \zeta, \xi$ between 0 and $\pi$,
where 
\begin{equation}\label{dne}
s=\sqrt{1-\alpha^2},\quad t=\sqrt{1-\gamma^2},
\end{equation}
$w_2$ is given as in~\eqref{grand}. The Hurwitz condition says
\begin{equation}\label{htz}
t\,\sigma_1 \cos(\psi)+s\,\sigma_2\cos(\phi)=-\alpha\nu-\gamma\mu,\quad s\,\sigma_1\cos(\zeta)+t\,\sigma_2\cos(\xi)=-\alpha\mu-\gamma\nu.
\end{equation}
If we set $\phi=0$ and the other three angles equal to $\pi/2$, then the Hurwitz condition is reduced to
$$
s\sigma_2=-\alpha\nu-\gamma\mu,\quad \alpha\mu+\gamma\nu=0,
$$
which is equivalent to
\begin{equation}\label{supp}
\sqrt{1-\alpha^2}\sqrt{1-\nu^2}=(\gamma^2-\alpha^2)\nu/\alpha.
\end{equation}
The equation does carry solutions; for instance, one can set
$\nu=\epsilon\alpha$. Then, we are solving
\begin{equation}\label{EQ}
\epsilon^2\;\gamma^4-2\epsilon^2\alpha^2\;\gamma^2+(1+\epsilon^2)\alpha^2-1=0,
\end{equation}
which has a solution 
$$
\gamma^2=\alpha^2+s\sqrt{\frac{1}{\epsilon^2}-\alpha^2}
$$
for $0<\gamma^2<1$ as long as we choose $\epsilon$ such that
$$
1<\epsilon^2< 1/\alpha^2,\quad \epsilon>0,
$$
The upshot is that we have now $p=11$ under the generic assumption that $\sigma_1\sigma_2\neq 0$. 

To prove the generic moduli dimension of type $[3,4,11]$ is 8,
observe that we can perturb slightly the angle $ \phi$ away from zero, then~\eqref{EQ} is perturbed into
$$
\epsilon^2\;\gamma^4-2\epsilon^2\alpha^2\;\gamma^2+(1+\epsilon^2)\alpha^2-1=s^2\sigma_2^2\sin^2(\phi),
$$
so that we obtain
\begin{equation}\label{haha}
\gamma^2=\alpha^2+s\sqrt{\frac{1}{\epsilon^2}-\alpha^2+\sigma_2^2\sin^2(\phi)},
\end{equation}
and so we have solutions for $\gamma$ for sufficiently small $\phi$. It follows that the analytic map
$$
\Pi_1: \tilde{F}_2\mapsto \omega_1,
$$
where $\omega_1$ is defined in~\eqref{pi}, from ${\mathbb R}^9$ to ${\mathbb R}$ is surjective around the image point 0 (we allow negative $\phi$). Moreover, since $\phi$ and $s$ are independent variables in~\eqref{haha}, it is then clear that 0 is a regular value of $\Pi_1$. Hence, by the rank theorem $\Pi_1^{-1}(0)$ is of dimension 8.

Similarly, we can let $\phi=\psi=0$ and the other two angles equal $\pi/2$. Then the Hurwitz condition reads
$$
t\sigma_1+s\sigma_2=-\alpha\nu-\gamma\mu,\quad \alpha\mu+\gamma\nu=0.
$$
Solving $\mu$ from the second equation and substituting it into the first yields
\begin{equation}\label{EQ1}
\sqrt{1-\mu^2}=\frac{-(\gamma^2-\alpha^2)\nu/\alpha-\sqrt{1-\alpha^2}\sqrt{1-\nu^2}}{\sqrt{1-\gamma^2}}.
\end{equation}
There is a solution for $0<\mu^2<1$ if we choose appropriately $\alpha$ and $\nu$ close to 1 and $\gamma$ close to zero. So, $p=10$ under the generic condition that $\sigma_1\sigma_2\neq 0$.

Similar to the preceding case,~\eqref{EQ1} is perturbed to
$$
\sqrt{1-\mu^2}=\frac{-(\gamma^2-\alpha^2)\nu/\alpha-\sqrt{1-\alpha^2}\sqrt{1-\nu^2}\cos(\phi)}{\sqrt{1-\gamma^2}\cos(\psi)}
$$
with solution for sufficiently small $\phi$ and $\psi$. Once more, since $s, t, \phi, \psi$ are independent variables and we allow negative angles, the map
$\Pi_2:\tilde{F}_2\mapsto (\omega_1,\omega_2)$ assumes $(0,0)$ as a regular value, so that $\Pi_2^{-1}((0,0))$ is of dimension 7.
 
In the same way, when we let $\phi=\psi=\zeta=0$ and $\xi=\pi/2$, the Hurwitz condition reads
\begin{equation}\label{EQ2}
t\sigma_1+s\sigma_2=-\alpha\nu-\gamma\mu,\quad t\sigma_2=-\alpha\mu-\gamma\nu,
\end{equation}
which can be put in the form by solving for $\sigma_1$ and $\sigma_2$,
$$
\sqrt{1-\mu^2}=(-\gamma/t+s\alpha/t^2)\mu+(-\alpha/t+s\gamma/t^2)\nu,\quad \sqrt{1-\nu^2}=-(\alpha/t)\mu -(\gamma/t)\nu.
$$
Each of these two equations represents a tilted ellipse in the $(\mu,\nu)$-plane centered at the origin. We can let $\alpha$ and $\gamma$ be sufficiently small so that the rotational angle, from the $u$-axis,
of the first ellipse  is small, while the length of its semi-major axis is close to 1 and of its semi-minor axis small. Meanwhile, the second ellipse has a rotational angle close to $\pi/2$ with similar lengths of semi-major and semi-minor axes. It follows that these two ellipses intersect at four points within the unit circle. Hence, $p=9$ under the generic assumption that $\sigma_1\sigma_2\neq 0$. 

Similar to the preceding case,~\eqref{EQ2} is perturbed to 
$$
t\,\sigma_1 \cos(\psi)+s\,\sigma_2\cos(\phi)=-\alpha\nu-\gamma\mu,\quad s\,\sigma_1\cos(\zeta)+t\,\sigma_2=-\alpha\mu-\gamma\nu.
$$
The above argument with ellipses asserts that as long as we keep the three angles small, we do have solutions with $s,t,\phi,\psi,\zeta$ as independent variables, so that once more $\Pi_3:\tilde{F}_2\mapsto (\omega_1,\omega_2,\omega_3)$ assumes $(0,0,0)$ as a regular value and so $\Pi_3^{-1}((0,0,0))$ is of dimension 6.

In summary, generic moduli dimension of type $[3,4,12-n]$ is $9-n$ for $0\leq n\leq 4$.

The situation is considerably simplified when $\sigma_1=0$ and $\sigma_2\neq 0$. In~\eqref{tai}, we can introduce a rotation on the plane spanned by $u_{-3}$ (the second column)  and $u_1$ (the fifth column) with the new basis vectors $u_{-3}^*, u_1^*$, where
$$
u_{-3}=\cos(\psi)\,u_{-3}^*+\sin(\psi)\, u_1^*,\quad u_1=-\sin(\psi)\, u_{-3}^*+\cos(\psi)\,u_1^*,
$$
relative to which the second column of $c_2$ is now zero; note that since the corresponding columns of $c_1$ is zero because $\sigma_1=0$, the process does not change anything else.  Similarly, we can cancel the third column by the eighth. After the cancellation, the nontrivial term in the fifth column becomes $t$ and the nontrivial term for the eighth column becomes $s$, with the second and the third columns zero. We thus obtain, up to range equivalence,
\begin{equation}\label{ta}
\aligned
&\begin{pmatrix}e_2&c_2\end{pmatrix} =\\
&\begin{pmatrix}s\sin(\phi)&0&0&0&0&s\cos(\phi)&0&0\\0& 0&0&0&t&0&0&0\\0&0&0&0&0&0&0&s\\0&0&0&t\sin(\xi)&0&0&t\cos(\xi)&0\end{pmatrix}.
\endaligned
\end{equation}
The Hurwitz condition~\eqref{htz} now gives the constraint
\begin{equation}\label{end}
s\, \sigma_2 \cos(\phi)=-\alpha\nu-\gamma\mu,\quad t\, \sigma_2 \cos(\xi)=-\alpha\mu-\gamma\nu,\quad \mu^2=1,\;\;\sigma_2=\sqrt{1-\nu^2}.
\end{equation}

Assume $st\neq 0$, By the angle being generic we mean the angle is neither 0 nor $\pi$. Then $p=10$ when $\phi, \xi$ are generic, since the 4-by-6 $c_2$ is of rank 4, and the other 6 dimensions come from $w_3$ and $c_1$; the moduli dimension is 3. In fact, up to range equivalence, it is a degenerate case of the 7-dimensional moduli discussed below~\eqref{EQ1}. 
On the other hand, if $\phi=0$ or $\pi$ then $\xi=0$ or $\pi$, because, with $\mu^2=1$, we can verify via~\eqref{end} that $s^2\sigma_2^2=(\alpha\nu+\gamma\mu)^2$ if and only if
$t^2\sigma_2^2=(\alpha\mu+\gamma\nu)^2;$ it follows that $p=8$, which can be brought to the degenerate $[3,4,8]$ case in Section~\ref{most} by range equivalence; the moduli dimension is 2.

Assume $s\neq 0$ and $t=0$. We have $\gamma^2=1 (=\mu^2)$, so that $\alpha^2=\nu^2$ by the second identity in~\eqref{end}, and so the first identity results in $\cos(\phi)=\pm 1$; thus $\sin(\phi)=0$. $p=7$ and the moduli dimension is 1. It can be brought to the degenerate $[3,4,7]$ type in Section~\ref{743} by range equivalence.
A similar conclusion holds for the case when $s=0$ and $t\neq 0$.

Assume $s=t=0$.  Then~\eqref{dne} and~\eqref{end} implies $\alpha^2=\gamma^2=\mu^2=\nu^2=1$, so that $\sigma_2=0$, a contradiction.

A parallel argument takes care of the case when $\sigma_1\neq 0$ and $\sigma_2=0$.

Lastly, when $\sigma_1=\sigma_2=0$,~\eqref{end} gives
$$
\alpha\mu+\gamma\nu=0,\quad \alpha\nu+\gamma\mu=0, \quad \mu^2=\nu^2=1,
$$
so that we may assume 
$$
\alpha=\gamma,\quad \mu=-\nu=1;
$$
in particular $s=t$. As above, we can cancel the first column in~\eqref{ta} by the sixth one, and the fourth by the seventh, so that, up to range equivalence, 
\begin{equation}\label{TA}
\aligned
&\begin{pmatrix}e_2&c_2\end{pmatrix} =\\
&\begin{pmatrix}0&0&0&0&s&0&0&0\\0&0&0&0&0&s&0&0\\0&0&0&0&0&0&s&0\\0&0&0&0&0&0&0&s\end{pmatrix}.
\endaligned
\end{equation}
We have $p=8$ if $s\neq 0$, 
which can be brought to the degenerate $[3,4,8]$ case in Section~\ref{most} by range equivalence. We come down to the quaternion multiplication when $s=0$.

The above analysis leads us to the following.

\begin{corollary} There are no orthogonal multiplications of type $[3,4,p]$ when $p=5$ or $6$.
\end{corollary}

\begin{remark} In fact, the corollary also follows from the general results discussed in~\cite[p. 409]{To1}. There is no orthogonal multiplications of
type $[3,4,5]$, since it can be restricted to type $[3,4,4]$. Moreover, the type $ [3,4,6]$ does not exist either, since it can be extended to
type $[3,6,6]$, and the Hurwitz-Rodon function $\rho(6) = 2< 3$, so that the type $[3,4,6]$
cannot be attained. 
\end{remark}

\end{document}